\pdfoutput=1

\documentclass[authoryear,preprint,nopreprintline,3p,times,11pt]{elsarticle}

\usepackage{amsmath, amsfonts,amssymb,amscd, amsthm}
\usepackage{graphicx}%
\usepackage{multirow}
\usepackage{textcomp}%
\usepackage{manyfoot}%
\usepackage{booktabs}%
\usepackage{listings}%
\usepackage{url}            
\usepackage{nicefrac}       
\usepackage{mathtools}
\usepackage{enumerate}
\usepackage{enumitem}
\usepackage{mathrsfs}
\usepackage{xcolor}
\usepackage{listings}
\usepackage{hyperref}
\usepackage{algorithmicx}
\usepackage[ruled]{algorithm}
\usepackage{algpseudocode}
\usepackage{ragged2e}
\usepackage{array}
\usepackage{float}
\usepackage{comment}
\usepackage{pgfplotstable}
\usepackage{bm}

\pgfplotsset{
        compat=1.9,
        xmin=0,
        /tikz/font=\footnotesize,
    }
\usepgfplotslibrary{colorbrewer,dateplot}
\pgfplotsset{compat=1.8}
\pgfmathdeclarefunction{fpumod}{2}{%
    \pgfmathfloatdivide{#1}{#2}%
    \pgfmathfloatint{\pgfmathresult}%
    \pgfmathfloatmultiply{\pgfmathresult}{#2}%
    \pgfmathfloatsubtract{#1}{\pgfmathresult}%
    \pgfmathfloatifapproxequalrel{\pgfmathresult}{#2}{\def\pgfmathresult{5}}{}%
}
\usepackage{tikz}
\usetikzlibrary{arrows,calc,decorations.pathreplacing,shapes.geometric,topaths,pgfplots.statistics,pgfplots.groupplots,pgfplots.colorbrewer,math}
\colorlet{fblue}{blue!30!white}
\colorlet{dbrown}{brown!60!black}
\colorlet{fred}{red!30!white}
\colorlet{fbrown}{brown!30!white}
\definecolor{mygreen}{rgb}{0.42, 0.56, 0.14}
\colorlet{lblue}{blue!80!white}
\colorlet{lbrown}{brown!90!black}
\colorlet{lred}{red!80!white}
\definecolor{deepg}{HTML}{5b6e5b}

\newcommand{\calU}{\mathcal{U}}
\newcommand{\calP}{\mathcal{P}}
\newcommand{\phiprob}{\Phi^{-1}(\rho)}

\alglanguage{pseudocode}

\newtheorem{prop}{Proposition}
\newtheorem{thm}{Theorem}

\newtheorem{exam}{Example}
\newlist{casesp}{enumerate}{1} 
\setlist[casesp]{align=left, 
                 listparindent=\parindent, 
                 parsep=\parskip, 
                 font=\normalfont, 
                 leftmargin=1em, 
                 labelindent=0em,
                 labelwidth=0pt, 
                 itemsep=0.5em,
                 itemindent=1em,labelsep=0.5em, 
                 partopsep=1em, 
				label=$\bullet$ Case~\arabic*.,ref=\arabic*
                 }

\begin{document}
\begin{frontmatter}
\title{Non-convex relaxation and 1/2-approximation algorithm for the chance-constrained binary knapsack problem}

\author[label1]{Junyoung Kim}
\ead{xhxhzld@snu.ac.kr}

\author[label1]{Kyungsik Lee\corref{cor1}}
\ead{optima@snu.ac.kr}
		
\cortext[cor1]{Corresponding author}
\address[label1]{Department of Industrial Engineering, Seoul National University, 1 Gwanak-ro, Gwanak-gu, Seoul 08826, Republic of Korea}

\begin{abstract}
We consider the chance-constrained binary knapsack problem (CKP), where the item weights are independent and normally distributed. We introduce a continuous relaxation for the CKP, represented as a non-convex optimization problem, which we call the non-convex relaxation. A comparative study shows that the non-convex relaxation provides an upper bound for the CKP, at least as tight as those obtained from other continuous relaxations for the CKP. Furthermore, the quality of the obtained upper bound is guaranteed to be at most twice the optimal objective value of the CKP. Despite its non-convex nature, we show that the non-convex relaxation can be solved in polynomial time. Subsequently, we proposed a polynomial-time 1/2-approximation algorithm for the CKP based on this relaxation, providing a lower bound for the CKP. Computational test results demonstrate that the non-convex relaxation and the proposed approximation algorithm yields tight lower and upper bounds for the CKP within a short computation time, ensuring the quality of the obtained bounds. 
\end{abstract}

\end{frontmatter}

\section{Introduction}\label{chap4:sec:intro}


The binary knapsack problem (BKP) is a widely studied combinatorial optimization problem. In this problem, we are given a set of items $N=\{1,...,n\}$, where each item $j\in N$ has a profit $c_j>0$ and a weight $a_j>0$. The objective is to select a subset of items that can be packed into a knapsack of a given capacity $b>0$, aiming to maximize the total profit of the selected items. The BKP holds both theoretical and practical significance in integer programming. Its combinatorial structure, defined with general coefficients, provides insights for analyzing complexity status and developing solution approaches for complicated integer programs. Specifically, the BKP frequently arises as a sub-problem of several large-scale integer programs, such as column generation sub-problems for generalized assignment problems \citep{savelsbergh1997branch} or cutting-stock problems \citep{gilmore1961linear}. Developing an efficient solution approach for the BKP is essential for successfully solving such large-scale integer programs. While the BKP is known to be NP-hard, it can be solved in pseudo-polynomial time using a dynamic programming algorithm. In addition to exact algorithms such as the dynamic programming algorithm, efficient approximation or heuristic algorithms have been developed over the last few decades. Comprehensive reviews of these studies can be found in \cite{martello1990knapsack} and \cite{kellerer2004}.




Given the importance of the BKP, recent research has focused on its variants that have emerged in challenging real-world optimization problems. Data uncertainty is a significant concern in practical optimization problems, impacting the quality and feasibility of solutions \citep{ben2009robust}. Various modeling frameworks have been developed to address this issue and incorporate data uncertainty into optimization models. Two representative frameworks are robust optimization \citep{ben2009robust} and stochastic optimization approaches \citep{birge2011introduction}. The former assumes that the data is drawn from a given uncertainty set and aims to find the best solution feasible for all realizations of the data. The latter approach utilizes stochastic information about the data and typically optimizes the expectation of the objective value. The chance-constrained programming approach introduced by \cite{charnes1958cost} is a type of stochastic optimization approach. Its primary goal is to obtain the best solution satisfying the constraints within a specified probabilistic threshold. Uncertain variants of the BKP within these frameworks, such as the robust knapsack problem \citep{yu1996max, iida1999note, bertsimas2004price} and the stochastic knapsack problem \citep{kosuch2010upper, merzifonluouglu2012static}, have also been extensively explored along with their solution approaches.

We consider the chance-constrained binary knapsack problem (CKP), where each item $j\in N$ has an uncertain item weight $\tilde{a}_j$. The objective of the CKP is to find a subset of items maximizing the total profit, while the probability of satisfying the capacity constraint is greater than or equal to a given threshold $\rho$. Here, we assume that $\rho\geq 0.5$. The CKP can be formulated as follows: 
\begin{align}
    \max\quad &\sum_{j \in N}c_j x_j \nonumber \\
    \text{s.t} \quad & \mathbb{P}\left\{\sum_{j \in N}\tilde{a}_j x_j \leq b\right\}\geq \rho \label{chap4:chancon} \\
    &  x \in \{0,1\}^n \nonumber
\end{align}
The feasible solution set of the CKP changes depending on the distribution of $\tilde{a}_j$'s due to the chance constraint \eqref{chap4:chancon}. This study assumes that the item weights are mutually independent and normally distributed with mean $a_j$ and standard deviation $\sigma$, i.e., $\tilde{a}_j \sim \mathcal{N}(a_j, \sigma_j^2)$. In this case, the chance constraint \eqref{chap4:chancon} can be reformulated as the following second-order cone constraint:
\begin{equation}
    \sum_{j \in N}a_j x_j + \phiprob\sqrt{\sum_{j \in N}\sigma_j^2 x_j^2} \leq b, \label{chancon_convex},
\end{equation}
where $\Phi^{-1}(\cdot)$ denotes the inverse of the cumulative density function for the standard normal distribution. Consequently, the CKP can be formulated as the following second-order cone program, which we refer to as the ISP.
\begin{equation*}
    \text{ISP: }\max\left\{\sum_{j \in N}c_j x_j:\eqref{chancon_convex}, x\in \{0,1\}^n\right\}
\end{equation*}
Without loss of generality, we assume that $\sum_{j \in N}a_j + \phiprob \sqrt{\sum_{j\in N}\sigma_j^{2}}> b$ and $a_j+\phiprob\sigma\leq b$ for each $j\in N$. 

The CKP may seem artificial due to the normality assumption on the probability distribution. However, the CKP can represent other BKP variants under different assumptions on the uncertain item weights. For example, when only the uncertainty set is known as an ellipsoidal set, the corresponding robust knapsack problem can be formulated as a form of the ISP \citep{ben1998robust}. In addition, when the uncertain item weights are known only through their first and second moments, the chance constraint \eqref{chap4:chancon} under the worse-case probability distribution can be represented as the constraint \eqref{chancon_convex} where $\phiprob=\sqrt{\rho/(1-\rho)}$ \citep{ghaoui2003worst,bertsimas2005optimal}. Therefore, studies on the CKP can be applied to these BKP variants.



Solution approaches for the CKP have been studied in the literature. The CKP is a generalization of the BKP, where the latter is a special case with $\sigma_j=0$ for each $j\in N$. Therefore, the CKP is also NP-hard. While an optimal solution for the CKP with a few items can be obtained by solving the ISP using the branch-and-bound algorithm provided by commercial optimization software, solving large-sized CKPs remains challenging. Scalable exact solution approaches, such as the dynamic programming algorithm for the BKP, have not yet been proposed. However, several heuristic algorithms have been developed for the CKP. 

\cite{shabtai2018relaxed} proposed a FPTAS (fully polynomial-time approximation scheme) for the CKP in a relaxed sense. For a given $\epsilon>0$, this scheme provides a feasible solution for the relaxed CKP where $b$ is replaced by $(1+\epsilon)b$, with the corresponding objective value being greater than or equal to $(1-5\epsilon)$ times the optimal objective value of the original CKP. 
\cite{han2016robust} considered a relaxation of the CKP, represented as a robust knapsack problem under a polyhedral uncertainty set. The authors proposed a pseudo-polynomial time algorithm for solving this problem and showed that the optimal solution satisfies the capacity constraint of the CKP with a probability of at least $\rho-\epsilon$ for a given $\epsilon>0$. 
We note that the solution approaches proposed by \cite{shabtai2018relaxed} and \cite{han2016robust} may yield an infeasible solution for the CKP. However, they can be used as a heuristic approach by applying them to the CKP instances modified appropriately.
\cite{joung2020robust} proposed another robust optimization-based heuristic for the CKP. The authors approximated the CKP as a robust knapsack problem under a box uncertainty set with a cardinality constraint. The authors then proposed a pseudo-polynomial time heuristic algorithm for the CKP, which solves this robust knapsack problem iteratively while adjusting the value of the cardinality. 

These solution approaches, however, may still involve a considerable amount of computation due to their high computational complexity. Furthermore, the quality of the obtained solutions is not guaranteed. A PTAS for the CKP was proposed by \cite{goyal2010ptas} using the parametric linear programming reformulation, which provides a $(1-3\epsilon)$-approximate solution for a given $\epsilon>0$. However, the PTAS also has a prohibitively large running time, which is undesirable in practice.

On the other hand, various formulations for the CKP have been introduced in the literature. \cite{atamturk2008polymatroids} formulated the CKP as an integer linear program using a submodular polytope \citep{cornuelsols2014IP}. \cite{goyal2010ptas} introduced an integer non-convex program for the CKP by replacing $x_j^2$ with $x_j$ in the ISP, based on the binarity of the decision variables. These formulations, including the ISP, yield continuous relaxations for the CKP, and their optimal solutions provide upper bounds for the problem. Typically, bounds obtained from relaxations of integer programming problems can be used to evaluate the quality of solutions obtained from heuristic algorithms. They also play a crucial role in developing exact algorithms based on enumeration schemes, such as the branch-and-bound algorithm. Moreover, approximate solutions for integer programming problems can often be obtained from optimal solutions of their continuous relaxations. 

However, the CKP's continuous relaxations have been scarcely studied; only two studies have addressed this issue. \cite{goyal2010ptas} showed that the continuous relaxation of the ISP can have a large integrality gap. \cite{atamturk2009submodular} examined the special case of the continuous relaxation of the integer non-convex program introduced by \cite{goyal2010ptas}, where $c_j\leq n$ for each $j \in N$. The authors proved that this case can be solved in polynomial time. To the best of our knowledge, no studies have yet theoretically compared the continuous relaxations of the CKP, and proposed solution approaches that utilize them for the CKP.

In this paper, we consider the continuous relaxation of the integer non-convex program for the CKP introduced by \cite{goyal2010ptas}. This relaxation, denoted as the \textit{non-convex relaxation}, is defined as the following non-convex optimization problem:
\begin{equation*}
    \max_{x\in [0,1]^n}\left\{\sum_{j \in N}c_jx_j :g(x)\leq b\right\},
\end{equation*}
where $g(x)= \sum_{j \in N}a_j x_j + \phiprob\sqrt{\sum_{j \in N}\sigma_j^2 x_j}$. We conduct a comparative study on the non-convex relaxation with the other continuous relaxations, along with an analysis of the properties of its optimal solution. Subsequently, an efficient algorithm is proposed for solving this relaxation. We also show that an approximate solution for the CKP can be obtained from the non-convex relaxation. Our contributions can be summarized as follows.

\begin{itemize}
    \item We show that the non-convex relaxation provides an upper bound for the CKP, at least as strong as those provided by the other continuous relaxations presented in the literature. Furthermore, we prove that the upper bound is less than or equal to twice the optimal objective value of the CKP.
    \item Despite the non-convexity, we show that the non-convex relaxation can be solved in polynomial time. Subsequently, we propose a polynomial-time 1/2-approximation algorithm for the CKP based on the non-convex relaxation. 
    \item Through extensive computational tests, we demonstrate that the non-convex relaxation and the proposed approximation algorithm provide tight upper and lower bounds for the CKP within a short computation time.
\end{itemize}

The remainder of this paper is organized as follows. In Section \ref{chap4:sec:compare}, we introduce the other continuous relaxations for the CKP presented in the literature and compare them with the non-convex relaxation. An efficient algorithm for solving the non-convex relaxation is devised in Section \ref{chap4:sec:algorithm}. Section \ref{chap4:sec:approx} presents how an approximate solution for the CKP can be obtained from the non-convex relaxation. Computational test results are reported in Section \ref{chap4:sec:exp}. Finally, the concluding remarks are given in Section \ref{chap4:sec:con}.

\section{Comparison of the continuous relaxations for the CKP}\label{chap4:sec:compare}

We start this section by formally defining the other continuous relaxations derived from the CKP formulations introduced in Section \ref{chap4:sec:intro}. By relaxing the integrality condition in the ISP, the following relaxation can be obtained:
\begin{equation*}
\max\left\{\sum_{j \in N}c_j x_j: x \in \calP_C\right\},    
\end{equation*}
where $\calP_C = \left\{x\in[0,1]^{n}: \eqref{chancon_convex}\right\}$.
We refer to this relaxation as the \textit{convex relaxation} for the CKP because the feasible solution set $\calP_C$ is convex. The objective value of this relaxation, denoted as $z_C$, can be efficiently obtained since this relaxation corresponds to a second-order cone programming problem.


Let us define a set function $ F(S)= \sum_{j \in S}a_j + \phiprob\sqrt{\sum_{j\in S}\sigma^{2}_j}$. Another formulation for the CKP proposed by \cite{atamturk2008polymatroids} is based on the submodularity of $F(S)$ and defined as the following integer linear program:
\begin{align}
    \max \quad &\sum_{j \in N}c_j x_j \nonumber \\
    \text{s.t}\quad &\sum_{j\in N}\pi_j x_j \leq b,\; \forall \pi \in \text{ext}(\Pi) \label{modconst}\\
    &x\in \{0,1\}^n \nonumber,
\end{align}
where $\Pi =\left\{\pi \in \mathbb{R}^{n}: \sum_{j \in S}\pi_j \leq F(S),\; \forall S\subseteq N\right\}$, and $\text{ext}(\Pi)$ is the set of extreme points of $\Pi$. The continuous relaxation of this formulation can be described as 
\begin{equation*}
\max\left\{\sum_{j \in N}c_j x_j: x \in \calP_P\right\},    
\end{equation*}
where $\calP_P=\left\{x \in [0,1]^{n}: \eqref{modconst}\right\}$. $\calP_P$ is a polyhedron since $|\text{ext}(\Pi)|$ is finite. Hence, we refer to this relaxation as the \textit{polyhedral relaxation} for the CKP and denote its optimal objective value as $z_P$. Even though this relaxation has exponentially many constraints, $z_P$ can be obtained in polynomial time using the cutting-plane algorithm \citep{edmond1971}. 

In the following discussion, we compare the optimal objective values of the convex, polyhedral, and non-convex relaxations, which are upper bounds for the CKP. Additionally, we examine the integrality gaps of these relaxations by comparing their optimal objective values with the optimal objective value of the CKP.

\subsection{Comparison of the optimal objective values of the continuous relaxations}




In this section, we show that an upper bound the CKP provided by the non-convex relaxation is at least as strong as those obtained from the other continuous relaxations. We first introduce two properties of an optimal solution for the non-convex relaxation. Let $\calP_{NC}$ be the feasible solution set of the non-convex relaxation.

\begin{prop}\label{chap4:prop:pack}
Let $x^{\star}\in \calP_{NC}$ be an optimal solution for the non-convex relaxation. Then $g(x^{\star})=b$.
\end{prop}
\begin{proof}
    There exists $k\in N$ such that $x^{\star}_k<1$ by the assumption that $\sum_{j \in N}a_j + \phiprob \sqrt{\sum_{j\in N}\sigma_j^{2}}> b$. Now, suppose that $g(x^{\star})<b$. In this case, an improved solution can be obtained by increasing the values of $x^{\star}_k$ since $c_k>0$. However, this result contradicts the optimality of $x^{\star}$. Therefore, the result follows. \qed
\end{proof}

\begin{prop}\label{chap4:prop:extreme}
There exists an optimal solution for the non-convex relaxation in which at most one variable has a fractional value.
\end{prop}
\begin{proof}
    Let $\hat{x} \in \calP_{NC}$ be an optimal solution of the non-convex relaxation, and suppose that $\hat{x}$ has more than two fractional variables. We show that another optimal solution can be constructed from $\hat{x}$, where one of the fractional variables becomes $0$ or $1$. Let $k$ and $l$ be the indices of some two fractional variables, where $0<\hat{x}_k<1$ and $0<\hat{x}_l <1$. Let us consider the following optimization problem with two variables.
    \begin{align}
        \max \quad &\sum_{j \in N\setminus\{k,l\}}c_j \hat{x}_j + c_k x_k + c_l x_l \nonumber \\
        \text{s.t} \quad &\sum_{j \in N\setminus \{k,l\}}a_j \hat{x}_j +a_k x_k + a_l x_l + \phiprob\sqrt{\sum_{j \in N\setminus\{k,l\}}\sigma^{2}_j\hat{x}_j + \sigma^{2}_k x_k + \sigma^{2}_l x_l} = b\label{chap4:eq:subconst}\\
        &0\leq x_k, x_l \leq 1\nonumber 
    \end{align}
    It is clear that $(\hat{x}_k, \hat{x}_l)$ is an optimal solution to the above problem, with feasibility ensured by Proposition \ref{chap4:prop:pack}. We first show that there exists another optimal solution for the problem \eqref{chap4:eq:subconst}, where one of the two variables has an integral value.
    
    From the first constraint of the problem \eqref{chap4:eq:subconst}, we can express $x_l$ as a function of $x_k\in [0,1]$, denoted as $h(x_k)$. It is clear that $h(x_k)$ is a decreasing function for $x_k\in [0,1]$. Furthermore, 
    it can be easily shown that $h(x_k)$ is convex. The problem \eqref{chap4:eq:subconst} can be rewritten with $h(x_k)$ as follows.
    \begin{equation}
        \max\{c_k x_k + c_l h(x_k): 0 \leq h(x_k)\leq 1,\; 0\leq x_k \leq 1\} \label{chap4:prob:sub}
    \end{equation}
    Here, we leave out the constant term in the objective function of the problem \eqref{chap4:eq:subconst}. The problem \eqref{chap4:prob:sub} is a single-variable maximization problem where the feasible solution set is convex due to the monotonicity of $h(x_k)$ for $x_k\in [0,1]$. The objective function is also convex due to the convexity of $h(x_k)$. Therefore, an optimal solution $x^{\star}_k$ for the problem \eqref{chap4:prob:sub} can be found on the boundary of $x_k$, that is, $x^{\star}_k\in \{0,1\}$ or $h(x^{\star}_k)\in \{0,1\}$. Let $x_l^{\star}=h(x^{\star}_k)$. Then, $(x^{\star}_k,x^{\star}_l)$ is another optimal solution for the problem \eqref{chap4:eq:subconst}, which has fewer fractional variables compared to $(\hat{x}_k, \hat{x}_l)$.

    Now, another feasible solution for the non-convex relaxation can be constructed from $\hat{x}$ by replacing $\hat{x}_k$ and $\hat{x}_l$ with $x^{\star}_k$ and $x^{\star}_l$, respectively. The obtained solution is also optimal for the non-convex relaxation, which has fewer fractional variables, by the definition of $(x^{\star}_k,x^{\star}_l)$. By applying this procedure repeatedly until one or no variable has a fractional value, we can construct an optimal solution for the non-convex relaxation, which satisfies the condition in the statement. Therefore, the result follows.\qed
\end{proof}

Based on these two properties, we compare the non-convex and polyhedral relaxations.
\begin{prop}\label{chap4:prop:NCvsP}
$ z_{NC} \leq z_P$
\end{prop}
\begin{proof}
Let $x^{\star}\in \calP_{NC}$ be an optimal of the non-convex relaxation, which satisfies Proposition \ref{chap4:prop:extreme}. We show that $x^{\star} \in \calP_P$, hence, $ z_{NC} \leq z_P$. For some $T\subseteq N$ and $t\in N\setminus T$, $x^{\star}$ can be described as $x^{\star}_j=1$ for each $j\in T$, $0\leq x_t<1$, and $x_j=0$ for each $j \in N\setminus T\cup\{t\}$. For the sake of simplicity, let $T=\{1,...,t-1\}$. We can check whether $x^{\star}\in \calP_P$ by solving the following separation problem associated with $\calP_P$:
\begin{equation*}
    \eta=\max\left\{\sum_{j \in N}\pi_j x^{\star}_j:\sum_{j\in S}\pi_j \leq F(S), \; S\subseteq N\right\}.
\end{equation*}
If $x^{\star}\notin \calP_{P}$, then $\eta>b$ while $\eta\leq b$ implies that $x^{\star}\in \calP_{P}$. It is well-known that an optimal solution $\pi^\star$ of this separation problem is defined as $\pi^{\star}_j = F(S_j)-F(S_{j-1})$ for each $j\in N$ where $S_j =\{1,\ldots,j\}$ and $S_0=\emptyset$ \citep{edmond1971}. Then,
\begin{equation*}
    \eta = \sum_{j\in N}\pi^{\star}_j x^{\star}_j = F(S_{t-1}) + (F(S_{t})-F(S_{t-1}))x^{\star}_{t}.
\end{equation*}
Let $x^S\in \{0,1\}^n$ be the vector that represents a set $S\subseteq N$, where, for each $j \in N$, $x^S_j =1$ if $j\in S$, otherwise, $x^S_j =0$. Then, $F(S)=g(x^S)$ for all $S\subseteq N$, and $x^{\star} = (1-x^{\star}_{t})x^{S_{t-1}} + x^{\star}_{t}x^{S_{t}}$. Furthermore, $\eta$ can be rewritten as follows:
\begin{equation*}
    \eta=(1-x^{\star}_{t})g(x^{S_{t-1}}) + x^{\star}_{t}g(x^{S_t})
\end{equation*}
It can be easily shown that $g(x)$ is concave. Accordingly, we have the following result.
\begin{equation*}
    b \geq g(x^{\star})= g((1-x^{\star}_{t})x^{S_{t-1}} + x^{\star}_{t}x^{S_{t}})\geq (1-x^{\star}_{t})g(x^{S_{t-1}}) + x^{\star}_{t}g(x^{S_{t}}) = \eta
\end{equation*}
Therefore, $x^{\star}\in \calP_{P}$, and the result follows. \qed
\end{proof}

Next, we compare the convex relaxation and the polyhedral relaxation.
\begin{prop}\label{chap4:prop:PvsC}
$ z_{P} \leq z_C$
\end{prop}
\begin{proof}
    We define an ellipsoidal set $\calU$ as follows.
    \begin{equation*}
        \calU = \left\{u\in \mathbb{R}_+^{n}: \sum_{j \in N}\left( \frac{u_j - a_j}{\sigma_j}\right)^2 \leq (\phiprob)^2\right\}
    \end{equation*}
    Let us consider the following robust knapsack problem under the ellipsoidal uncertainty set $\calU$.
    \begin{align*}
        \max \quad &\sum_{j\in N}c_j x_j\\
        \text{s.t}\quad & \sum_{j \in N} u_j x_j \leq b,\; \forall u \in \calU\\
        & x \in \{0,1\}^n
    \end{align*}
    As mentioned in Section \ref{chap4:sec:intro}, this problem is equivalently reformulated as the ISP.
    Let $z_R$ and $\calP_R$ be the optimal objective value and the feasible solution set of the continuous relaxation of this problem, respectively. Then, $z_R=z_C$ and $\calP_R=\calP_C$. On the other hand, $\calP_P$ remains the same even if $\text{ext}(\Pi)$ is replaced with $\Pi$ in its definition. Based on this observation, we prove that $\calP_P\subseteq \calP_R=\calP_C$ by showing that $\calU\subseteq \Pi$. 

    For some $\hat{u}\in \calU$, we can see that 
    \begin{equation*}
        \sum_{j \in S}\hat{u}_j \leq \sum_{j \in S}a_j + \phiprob\sqrt{\sum_{j\in S}\sigma_j^2} = F(S),\;\forall S\subseteq N,
    \end{equation*}
    because
    \begin{equation*}
        \max\left\{\sum_{j \in S}u_j: u \in \calU\right\} = \sum_{j \in S}a_j + \phiprob\sqrt{\sum_{j\in S}\sigma_j^2}.
    \end{equation*}
    Therefore, $\hat{u}\in \Pi$. This result implies that $\calU\subseteq \Pi$ and $\calP_R \subseteq \calP_P$, that is, $z_P\leq z_R=z_C$. \qed
\end{proof}


Proposition \ref{chap4:prop:NCvsP} and \ref{chap4:prop:PvsC} directly lead to the following result.
\begin{thm}
    $z_{NC}\leq z_{P}\leq z_{C}$.
\end{thm}
Therefore, the upper bound for the CKP, provided by the non-convex relaxation is as strong as those provided by the other continuous relaxations. 

\subsection{Integrality gaps of the continuous relaxations}

We further analyze the tightness of the upper bounds provided by the continuous relaxations in terms of the integrality gaps. The integrality gap of each continuous relaxation is defined as $(z-z_{OPT})/z_{OPT}$ where $z$ is the optimal objective value of the continuous relaxation and $z_{OPT}$ is the optimal objective value of the CKP.

Let us consider the following example introduced in \cite{goyal2010ptas}.
\begin{exam}\label{exam1}
Consider the CKP instances where $c_j=\sigma_j=1$, and $a_j=1/\sqrt{n}$ for each $j\in N$ where $n\geq 56$. Additionally, let $b=3$ and $\phiprob=1.5$. Then, optimal solutions for these CKP instances should select 3 items. If 4 items are selected, the solution is not feasible for these instances because
\begin{equation*}
    \frac{4}{\sqrt{n}} + \phiprob\sqrt{4} = \frac{4}{\sqrt{n}} + 3 > 3.
\end{equation*}
Therefore, $z_{OPT}=3$.
\end{exam}
\cite{goyal2010ptas} showed that the integrality gap of the convex relaxation is $\Omega(\sqrt{n})$ through this example. Therefore, the convex relaxation may have a large integrality gap which increases as $n$ grows. 

We obtain the similar result for the polyhedral relaxation using Example \ref{exam1}.
\begin{prop}\label{chap4:prop:polygap}
    There exists CKP instances where the integrality gap of the polyhedral relaxation is $\Omega(\sqrt{n})$.
\end{prop}
\begin{proof}
    Let us consider $\hat{x}\in \mathbb{R}_+^n$ defined as $\hat{x}_j=1/\sqrt{n}$ for each $j\in N$. We first show that $\hat{x}$ is a feasible solution of the polyhedral relaxation for the CKP instances in Example \ref{exam1}. The separation problem for $\hat{x}$ is defined as follows.
    \begin{equation*}
        \eta=\max\left\{\sum_{j \in N}\pi_j \hat{x}_j:\sum_{j\in S}\pi_j \leq F(S), \; S\subseteq N\right\}.
    \end{equation*}
    Let $S_j=\{1,\ldots,j\}$ for each $j\in N$ where $S_0=\emptyset$. An optimal solution $\pi^{\star}$ of this separation problem can be constructed as 
    \begin{equation*}
        \pi_j^{\star} = F(S_j)-F(S_{j-1})=\frac{1}{\sqrt{n}}+\phiprob(\sqrt{j}-\sqrt{j-1}),    
    \end{equation*}
    for each $j\in N$. Then, we can see that
    \begin{equation*}
        \eta = \sum_{j \in N}\frac{1}{\sqrt{n}}\cdot\pi_j^{\star} = \frac{1}{\sqrt{n}}\cdot F(n) = 1+\phiprob =2.5.
    \end{equation*}
    Therefore, $\hat{x}\in \calP_P$ since $\eta<b=3$. On the other hand, the objective value of the polyhedral relaxation, corresponding to $\hat{x}$, is $\sqrt{n}$. These results imply that $z_P\geq \sqrt{n}$ and the integrality gap of the polyhedral relaxation is $(\sqrt{n}-3)/3$ for these instances. Therefore, the result follows. \qed
\end{proof}

On the contrary to the convex and polyhedral relaxations, the following theorem shows that the integrality gap of the non-convex relaxation is bounded by a constant factor.

\begin{thm}\label{chap4:prop:intgap}
$z_{NC}\leq 2z_{OPT}$
\end{thm}
\begin{proof}
Let $x^{\star} \in \calP_{NC}$ be an optimal solution for the non-convex relaxation, which satisfies Proposition \ref{chap4:prop:extreme}. We denote $t\in N$ as the index of the fractional variable in $x^{\star}$ such that $0\leq x^{\star}_t<1$. From $x^{\star}$, we can construct two feasible solutions $x^D$ and $x^U$ for the CKP as follows: $x^{D}_j=\lfloor x^{\star}_j\rfloor $ for each $j\in N$ and $x^{U}_t=1$ while $x^{U}_j=0$ for each $j\in N\setminus\{t\}$.
Then, we can see that
\begin{equation*}
    z_{NC}=\sum_{j \in N\setminus\{t\}}c_jx_j^{\star} + c_t x_t^{\star} \leq \sum_{j \in N}c_jx_j^{D} + c_t x_t^{U} \leq 2z_{OPT},
\end{equation*}
where the last inequality holds due to the feasibility of $x^{D}$ and $x^{U}$ for the CKP. Therefore, the result follows. \qed
\end{proof}
Theorem \ref{chap4:prop:intgap} implies that the integrality gap of the non-convex relaxation is always less than 1 since 
\begin{equation*}
    \frac{z_{NC}-z_{OPT}}{z_{OPT}} \leq  \frac{z_{OPT}}{z_{OPT}} =1.
\end{equation*}
Therefore, the non-convex relaxation provides a quality-guaranteed upper bound for the CKP. 

The difference between the non-convex relaxation and the other continuous relaxations can be definitely shown in Example \ref{exam1}. As shown in the proof of Proposition \ref{chap4:prop:polygap}, $z_{C}$ and $z_{P}$ are greater than or equal to $\sqrt{n}$ for the CKP instances described in Example \ref{exam1}. These results mean that $z_{C}$ and $z_{P}$ go to infinity as $n$ goes to infinity, while $z_{OPT}$ is always 3. However, $z_{NC}$ is always less than or equal to $6$ regardless of $n$ by Theorem \ref{chap4:prop:intgap}. 

In summary, the non-convex relaxation provides a quality-guaranteed upper bound for the CKP, which is at least as tight as the other continuous relaxations. To obtain such a bound, we should solve the non-convex relaxation represented as a non-convex optimization problem. In general, non-convex optimization problems are nontrivial to solve. However, in the following section, we show that an optimal solution of the non-convex relaxation can be obtained in polynomial time.

\section{Polynomial-time algorithm for the non-convex relaxation}\label{chap4:sec:algorithm}



We first introduce a parametric method that constructs feasible solutions of the non-convex relaxation. Utilizing these feasible solutions, we reformulate the non-convex relaxation as a single-variable optimization problem. We then devise a polynomial-time algorithm for solving the reformulated problem. 

\subsection{Reformulation of the non-convex relaxation}

Let us define a function $a_j(\delta)$ for $\delta\in \mathbb{R}_+$ as 
\begin{equation*}
    a_j(\delta)=
    \begin{cases}
        a_j  &,\;\text{if }\delta=0 \\
        a_j +\frac{\phiprob}{2\sqrt{\delta}}\sigma_j^{2} &,\;\text{otherwise}
    \end{cases},
\end{equation*}
for each $j\in N$ and $\bar{N}=\{j\in N:\sigma_j=0\}$, where $|\bar{N}|=\bar{n}$. We also define $p_j(\delta)=c_j/a_j(\delta)$ for each $j\in N$. 

Now, let $\tau(\delta)$ be a function representing a variable sequence of $N$ that varies with $\delta\in \mathbb{R}_+$. When $\delta>0$, $\tau(\delta)$ is defined as a variable sequence of $N$ sorted in descending order of $p_j(\delta)$'s. In $\tau(0)$, variables of $\bar{N}$ are sorted first in descending order $p_j(\delta)$'s, followed by variables in $N\setminus \bar{N}$ being sorted in descending order $p_j(\delta)$'s. Ties in $\tau(\delta)$ are resolved in descending order of $\sigma_j^2/c_j$'s, and if these values are still equivalent, further resolution takes place in ascending order of indices. Accordingly, $\tau(\delta)$ is uniquely defined for each $\delta \in \mathbb{R}_+$.

A feasible solution of the non-convex relaxation can be obtained in a greedy manner following $\tau(\delta)$ for each $\delta\in\mathbb{R}_+$. For the sake of simplicity, let $\tau(\delta)=(\tau_1,\ldots,\tau_n)$. We can easily find $t\in N$ such that 
$$t=\min\left\{k\in N:\sum_{j=1}^{k}a_{\tau_j} + \phiprob\sqrt{\sum_{j=1}^{k}\sigma_{\tau_j}^{2}}\geq b\right\}.$$ 
We note that such $t$ always exists for each $\delta>0$ due to the assumption that $\sum_{j \in N}a_j + \phiprob \sqrt{\sum_{j\in N}\sigma_j^{2}} > b$. Additionally, let $\theta^{\star}$ be a smaller value among two solutions for the following equation for $\theta\in \mathbb{R}$:
\begin{equation*}
   \phiprob^2 \left(\sum_{j=1}^{t-1}\sigma^{2}_{\tau_j} + \sigma^{2}_{\tau_t} \theta\right) =\left(b-\sum_{j=1}^{t-1}a_{\tau_j} - a_{\tau_{t}} \theta \right)^2,
\end{equation*}
which is derived from
\begin{equation*}
   \sum_{j=1}^{t-1}a_{\tau_j} + a_{\tau_{t}} \theta + \phiprob \sqrt{\sum_{j=1}^{t-1}\sigma^{2}_{\tau_j} + \sigma^{2}_{\tau_t} \theta} =b.
\end{equation*}
It is clear that $0 < \theta^{\star}\leq 1$. We define $x(\delta)$ as
\begin{equation*}
    x_{\tau_j}(\delta) =
    \begin{cases}
        1  &,\;j<t \\
        \theta^\star &,\;j=t \\
        0  &,\;j>t
    \end{cases}
\end{equation*}
for each $j\in N$. Then, $x(\delta)$ for each $\delta\in \mathbb{R}_+$ is a feasible solution for the non-convex relaxation since $g(x(\delta))=b$ and $x(\delta)\in [0,1]^n$ by definition. Furthermore, $x(\delta)$ has at most one fractional variable and can be constructed in $O(n\log n)$ for each $\delta\in \mathbb{R}_+$. 


Let $x^{\star}\in \calP_{NC}$ be an optimal solution for the non-convex relaxation, and $\delta^{\star}=\sum_{j \in N}\sigma_j^{2}x_j^{\star}$. We now show that $x(\delta^{\star})$ is an optimal solution for the non-convex relaxation. Let us first consider the following fractional knapsack problem:
\begin{align}
    \text{FK:}\quad \max \quad &\sum_{j \in N}c_j x_j \nonumber\\
    \text{s.t} \quad & \sum_{j \in N}a_j(\delta^{\star}) x_j \leq b-\frac{\phiprob}{2} \sqrt{\delta^{\star}} \label{FKconst1}\\
    & x_j = 0, \; j \in N\setminus \bar{N}\;\text{if }\delta^{\star}=0\label{FKconst2}\\
    &x \in [0,1]^n \nonumber
\end{align}
We denote the feasible solution set of the FK as $\calP_{FK}$. 


\begin{prop}\label{chap4:prop:surrogateknapsack}
    An optimal solution for the FK is optimal for the non-convex relaxation.
\end{prop}
\begin{proof}
    We first show that $x^{\star} \in \calP_{FK}$. If $\delta^{\star}=0$, then $x_j^{\star}=0$ for each $j\in N\setminus \bar{N}$ by the definition of $x^{\star}$, which means that $x^{\star}$ satisfies the constraints \eqref{FKconst2}. On the other hand, the following equations hold in this case:
    \begin{equation*}
        b= g(x^{\star}) = \sum_{j \in N}a_j(0)x_j^{\star},   
    \end{equation*}
    where $g(x^{\star})=b$ by Proposition \ref{chap4:prop:pack}. Therefore, $x^{\star}$ satisfies the constraint \eqref{FKconst1}, and  $x^{\star}\in \calP_{FK}$ when $\delta^{\star}=0$. Now, suppose that $\delta^{\star}>0$. $g(x^{\star})$ can be expressed as follows:
    \begin{align*}
        g(x^{\star}) =  \sum_{j \in N}a_j(\delta^{\star}) x_j^{\star} + \frac{\phiprob}{2}\sqrt{\delta^{\star}}
    \end{align*}
    Then, $g(x^{\star})=b$ by Proposition \ref{chap4:prop:pack} implies that $x^{\star}$ satisfies the constraint \eqref{FKconst1}. Therefore, $x^{\star}\in \calP_{FK}$ when $\delta>0$. This result implies that $z_{FK} \geq z_{NC}$ where $z_{FK}$ is the optimal objective value of the FK. 

    We next show that $\calP_{FK}\subseteq \calP_{NC}$, that is, $z_{FK} \leq z_{NC}$. 
    Let $\hat{x} \in \calP_{FK}$ be given when $\delta^{\star}=0$. Then, we can see that 
    \begin{equation*}
        g(\hat{x})=\sum_{j \in N}a_j\hat{x}_j    
    \end{equation*}
    by the constraints \eqref{FKconst2}, while the constraint \eqref{FKconst1} implies that $g(\hat{x})\leq b$. Therefore, $\hat{x}\in \calP_{NC}$. Now, suppose that $\hat{x} \in \calP_{FK}$ is given when $\delta^{\star}>0$. The constraint \eqref{FKconst1} can be arranged as follows:
    \begin{equation}
        \sum_{j \in N}a_j \hat{x}_j + \frac{\phiprob}{2}\left( 
        \frac{\hat{\delta}}{\sqrt{\delta^{\star}}} +\sqrt{\delta^{\star}}
        \right)
        \leq b, \label{subeq1}
    \end{equation}
    where $\hat{\delta}=\sum_{j \in N}\sigma_j^{2}\hat{x}_j$.
    We can obtain the following arithmetic–geometric mean inequality.
    \begin{equation}
        2\sqrt{\hat{\delta}}\leq  \frac{\hat{\delta}}{\sqrt{\delta^{\star}}} +\sqrt{\delta^{\star}} \label{subeq2}
    \end{equation}
    From the inequalities \eqref{subeq1} and \eqref{subeq2}, we can see that
    \begin{equation*}
        g(\hat{x})= \sum_{j \in N}a_j \hat{x}_j + \phiprob\sqrt{\hat{\delta}} \leq \sum_{j \in N}a_j \hat{x}_j + \frac{\phiprob}{2}\left( 
        \frac{\hat{\delta}}{\sqrt{\delta^{\star}}} +\sqrt{\delta^{\star}}
        \right)
        \leq b.  
    \end{equation*}
    Hence, $\hat{x}\in\calP_{NC}$ and $\calP_{FK}\subseteq \calP_{NC}$. 
    
    These results implies that $z_{FK}=z_{NC}$. Therefore, an optimal solution for the FK is not only feasible, but also optimal for the non-convex relaxation. \qed
\end{proof}

\begin{prop}\label{deltastaropt}
    $x(\delta^{\star})$ is an optimal solution for the FK.
\end{prop}
\begin{proof}
    Let $x^{FK}$ be an optimal solution for the FK.
    Suppose that $\delta^{\star}>0$. 
    $x^{FK}$ can be constructed by greedily selecting variables of $N$ with the highest ratios of $c_j/a_j(\delta^{\star})$ until the constraint \eqref{FKconst1} is satisfied as an equation \citep{dantzig1957discrete}. 
    $x^{FK}$ is an optimal solution for the non-convex relaxation by Proposition \ref{chap4:prop:surrogateknapsack}, and $g(x^{FK})=b$ by Proposition \ref{chap4:prop:pack}. 
    In other words, $x^{FK}$ is a feasible solution for the non-convex relaxation, constructed in a greedy manner following the sequence $\tau^{\star}$ until $g(x^{FK})=b$. This definition corresponds to that of $x(\delta^{\star})$, hence, $x^{FK}=x(\delta^{\star})$.
    Now, suppose that $\delta^{\star}=0$. In this case, $\sum_{j\in \bar{N}}a_j\geq b$ since 
    \begin{equation*}
        b=g(x^{\star}) = \sum_{j\in \bar{N}}a_jx_j^{\star}
    \end{equation*}
    by Proposition \ref{chap4:prop:pack}. $x^{FK}$ can then be constructed by selecting variables of $\bar{N}$ with the highest ratios of $c_j/a_j(0)$ until $g(x^{FK})=b$, where $x^{FK}_j=0$ for each $j\in N\setminus \bar{N}$. Therefore, $x^{FK}$ is equivalent to $x(0)$, and the result follows. \qed
\end{proof}

Proposition \ref{chap4:prop:surrogateknapsack} and \ref{deltastaropt} imply that $x(\delta^\star)$ is an optimal solution for the non-convex relaxation. Therefore, we have the following result due to the feasibility of $x(\delta)$ for the non-convex relaxation.
\begin{equation*}
\sum_{j \in N}c_j x_j(\delta)\leq \sum_{j \in N}c_j x_j(\delta^\star),\; \forall \delta \in \mathbb{R}_+
\end{equation*}
This observation allows us to reformulate the non-convex relaxation as the following single-variable optimization problem. 
\begin{align*}
    \text{NCR:}\quad \max\quad &\sum_{j \in N}c_j x_j(\delta)\\
    \text{s.t}\quad &\delta \in \mathbb{R}_+
\end{align*}
$\delta^{\star}$ is an optimal solution for the NCR, but it may not be the unique one. Nonetheless, solving the NCR can provide an optimal solution for the non-convex relaxation. The following section proposes a polynomial-time algorithm for solving the NCR.

\subsection{Polynomial-time algorithm for solving the reformulated non-convex relaxation}

Recall that $x(\delta)$ is determined by the variable sequence $\tau(\delta)$ for each $\delta\in\mathbb{R}_+$. In general, there may exist exponentially many different variable sequences. However, in the subsequent discussion, we prove that the number of different variable sequences indicated by $\tau(\delta)$ is polynomially bounded.


The following example illustrates that, when $\delta>0$, $\tau(\delta)$ does not vary with $\delta$ continuously.
\begin{exam}\label{exam2}
    Let $N=\{1,2,3\}$, $\phiprob=2$, $(c_1,a_1,\sigma^{2}_1)=(1,2,3)$, $(c_2,a_2,\sigma^{2}_2)=(1,3,1)$, and $(c_3,a_3,\sigma^{2}_3)=(1,2.5,1.5)$. When $\delta>0$, $p_1(\delta)$, $p_2(\delta)$, and $p_3(\delta)$ are defined as follows, and illustrated in Figure \ref{chap4:fig:p}.
    \begin{equation*}
        p_1(\delta) = \frac{1}{2+\frac{3}{\sqrt{\delta}}}, \quad p_2(\delta) = \frac{1}{3+\frac{1}{\sqrt{\delta}}}, \quad p_3(\delta) = \frac{1}{2.5+\frac{1.5}{\sqrt{\delta}}}
    \end{equation*}
    \begin{figure}[!th]
        \centering
        \begin{tikzpicture}[scale=1]
            \draw[->] (0, 1.5) -- (8, 1.5) node[below] {$\delta$};
            \draw[->] (0, 1.5) -- (0, 5.5) node[left] {$p_j(\delta)$};
            \draw[domain=0.1:8, samples=100, smooth, variable=\x, red, thick] plot ({\x}, {15/(2+3/sqrt(1.5*\x))});
            \draw[domain=0.03:8, samples=100, smooth, variable=\x, blue,thick,  dashdotted] plot ({\x}, {15/(3+1/sqrt(1.5*\x))});
            \draw[domain=0.03:8, samples=100, smooth, variable=\x, mygreen,ultra thick,dotted] plot ({\x}, {15/(2.5+1.5/sqrt(1.5*\x))});
            \node[mark size=2pt,color=red] at (0,1.5) {\nullfont\pgfuseplotmark{o}};
            \node[mark size=2pt,color=blue] at (0,1.5) {\nullfont\pgfuseplotmark{o}};
            \node[mark size=2pt,color=mygreen] at (0,1.5) {\nullfont\pgfuseplotmark{o}};
            \node[mark size=2pt,color=black] at (0.6667,3.75) {\nullfont\pgfuseplotmark{*}};
            \node[mark size=2pt,color=black] at (2.6667,4.2855) {\nullfont\pgfuseplotmark{*}};
            \node[mark size=2pt,color=black] at (6,5) {\nullfont\pgfuseplotmark{*}};

            \draw[dashed,black ] (0.6667,3.75) -- (0.6667,1.5) node[below] {$1$};
            \draw[dashed,black ] (2.6667,4.2855) -- (2.6667,1.5) node[below] {$4$};
            \draw[dashed,black ] (6,5) -- (6,1.5) node[below] {$9$};
            
            \draw[black, thin] (1.1,6) rectangle (6.7,5.4);
            \draw[thick, red] (1.3,5.7) to (1.9,5.7);
            \node[right=0.5em] at (1.9,5.7) {$p_1(\delta)$};
            \draw[thick, blue,dashdotted] (3.1,5.7) to (3.7,5.7);
            \node[right=0.5em] at (3.7,5.7) {$p_2(\delta)$};
            \draw[ultra thick, mygreen,dotted] (4.9,5.7) to (5.5,5.7);
            \node[right=0.5em] at (5.5,5.7) {$p_3(\delta)$};
            
        \end{tikzpicture}
    \caption{Illustration of $p_1(\delta)$, $p_2(\delta)$, and $p_3(\delta)$}
    \label{chap4:fig:p}
    \end{figure}
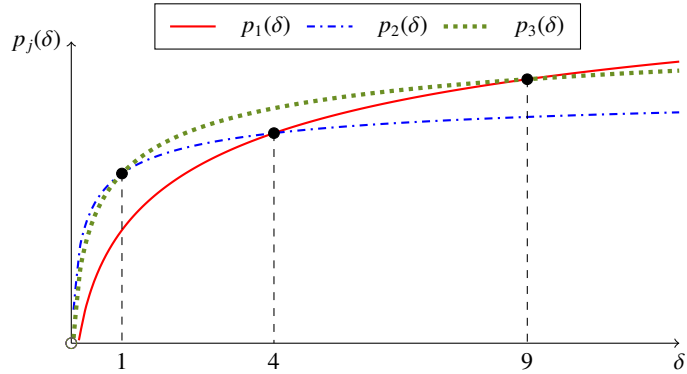
    We note that $p_2(1)=p_3(1)$, $p_1(4)=p_2(4)$, and $p_1(9)=p_3(9)$. $\tau(\delta)$ is defined as follows.
    \begin{equation*}
        \tau(\delta)=
        \begin{cases}
            (2,3,1) & \text{, if} \quad  0<\delta< 1\\
            (3,2,1) & \text{, if} \quad 1\leq \delta < 4\\
            (3,1,2) & \text{, if} \quad 4\leq \delta< 9\\
            (1,3,2) & \text{, if} \quad 9\leq \delta
        \end{cases}.
    \end{equation*}
    Here, $\tau(1)=(3,2,1)$ since $\sigma_2/c_2=1 < \sigma_3/c_3=1.5$ by definition, although $p_2(1) = p_3(1)$. The same applies to $\tau(4)$ and $\tau(9)$.
\end{exam}

$\tau(\delta)$ changes as $\delta$ increases only when the relative order of some two variables in $\tau(\delta)$ is reversed. For instance, in Example \ref{exam2}, variable 1 follows variable 3 in $\tau(\delta)$ when $0<\delta< 9$, while variable 1 is followed by variable 3 when $9\leq \delta$. Therefore, $\tau(\delta)$ changes at $\delta=9$. We term the values of $\delta$ where such reversals occur as \textit{reverse points} for those two variables. Formally, for a pair of variables $\{k,l\}\subseteq N$, a reverse point for $\{k,l\}$ is defined as $q>0$ such that $p_k(q)=p_l(q)$ and the relative order of $k$ and $l$ changes between $\tau(q-\alpha)$ and $\tau(q)$ for a sufficiently small $\alpha>0$. This observation enables us to enumerate all the different variable sequences indicated by $\tau(\delta)$ by investigating all the reverse points for every pair of variables.


We first show that there exists at most one reverse point for each pair of variables.
\begin{prop}\label{prop:orderchange}
    For some $\{k,l\}\subseteq N$ such that $\sigma_k^2/c_k \leq \sigma_l^2/c_l$, a reverse point uniquely exists if and only if
    \begin{equation}
        \frac{\sigma_k^2}{c_k} < \frac{\sigma_l^2}{c_l}, \;\frac{a_k}{c_k} > \frac{a_l}{c_l} \label{uniquecondition}.
    \end{equation}
    If a reverse point $q$ exists, then $k$ is followed by $l$ in $\tau(\delta)$ when $0<\delta< q$, while $k$ follows $l$ in $\tau(\delta)$ when $q\leq \delta$.
\end{prop}
\begin{proof}
    By definition, a reverse point of $\{k,l\}$ is a solution of $p_k(\delta)=p_l(\delta)$ which can be rewritten as follows.
    \begin{equation}
        \left(\frac{\sigma_k^2}{c_k} -\frac{\sigma_l^2}{c_l}\right)\frac{\phiprob}{2\sqrt{\delta}} = \frac{a_l }{c_l} - \frac{a_k}{c_k} \label{eqchange}
    \end{equation}
    The condition for this equation to have a solution can be categorized into the following cases.
    \begin{casesp}
        \item \eqref{eqchange} is valid for all $\delta>0$ when
        \begin{equation*}
            \frac{\sigma_k^2}{c_k} = \frac{\sigma_l^2}{c_l},\;\frac{a_l}{c_l} = \frac{a_k}{c_k}
        \end{equation*}
        \item \eqref{eqchange} has a unique solution $q$ when \eqref{uniquecondition}
    \end{casesp}
    In Case 1, no reverse point exists for $\{k,l\}$ since the relative order of $k$ and $l$ in $\tau(\delta)$ remains the same for all $\delta>0$. On the other hand, in Case 2, the relative order of $k$ and $l$ changes once at $\delta=q$, as illustrated in Figure \ref{figcase}. 
    \begin{figure} [!th]
    \centering
    \begin{tikzpicture}[scale=0.5]
        \draw[->] (0, 0) -- (7, 0) node[below] {$\delta$};
        \draw[->] (0, 0) -- (0, 6.5) node[left] {$p_j(\delta)$};
        \draw[domain=0.01:7, samples=100, smooth, variable=\x, red, thick] plot ({\x}, {15/(2+3/sqrt(1.5*\x))});
        \draw[domain=0.01:7, samples=100, smooth, variable=\x, blue, ultra thick,  dotted] plot ({\x}, {15/(1+5/sqrt(1.5*\x))});
        \node[mark size=2pt,color=red] at (0,0) {\pgfuseplotmark{o}};
        \node[mark size=2pt,color=blue] at (0,0) {\pgfuseplotmark{o}};
        \node[mark size=2pt,color=black] at (2.6667,4.2855) {\pgfuseplotmark{*}};
        \draw[dashed,black ] (2.6667,4.2855) -- (2.6667,0) node[below] {$q$};
        \draw[black, thin] (0.3,7.5) rectangle (7.3,6.3);
            \draw[thick, red] (0.5,6.9) to (1.3,6.9);
            \node[right=0.5em] at (1.3,6.9) {$p_k(\delta)$};
            \draw[ultra thick, blue,dotted] (4.3,6.9) to (5.1,6.9);
            \node[right=0.5em] at (5.1,6.9) {$p_l(\delta)$};
    \end{tikzpicture}
    \caption{Shapes of $p_k(\delta)$ (solid) and $p_l(\delta)$ (dotted) in Case 2}\label{figcase}	
    \end{figure}
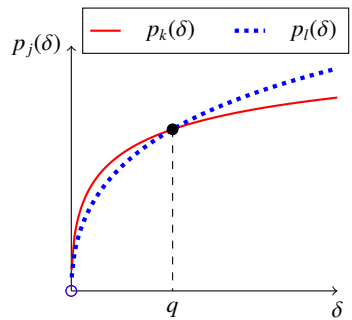
    Therefore, $q$ is the unique reverse point for $\{k,l\}$. In other words, a reverse point exists if and only in Case 2, and such reverse point is unique. Furthermore, as shown in Figure \ref{figcase}, $p_k(\delta) > p_l(\delta)$ when $0<\delta< q$, while $p_k(\delta)\leq p_l(\delta)$ when $q\leq \delta$. Then, $k$ is followed by $l$ in $\tau(\delta)$ when $0<\delta< q$, while $k$ follows $l$ in $\tau(\delta)$ when $q<\delta$. We note that $l$ is followed by $k$ in $\tau(q)$ since $\sigma_k^2/c_k < \sigma_l^2/c_l$ in this case. Therefore, the result follows. \qed
\end{proof}

By Proposition \ref{prop:orderchange}, there exist $O(n^2)$ reverse points for every pair of variables. We now show that all different variable sequences represented by $\tau(\delta)$ can be obtained by investigating such reverse points.


Let $Q$ be the set of reverse points for every pair of variables. For the sake of simplicity, let $Q=\{\delta_1,\ldots,\delta_m\}$ where $\delta_1\leq\cdots\leq \delta_m$ and $M=\{1,\ldots,m\}$. Subsequently, we define $\Delta = \{0,\gamma\}\cup Q$, where $0<\gamma<\delta_1$. Of course, $Q$ may be empty. In this case, we define $\gamma$ as any positive real number. 

\begin{prop}\label{chap4:prop:equiv}
    For each $\delta\in \mathbb{R}_+$, there exists $q\in \Delta$ such that $\tau(\delta) = \tau(q)$.
\end{prop}
\begin{proof}
    We only consider the case when $\delta>0$ since it is clear when $\delta=0$. Suppose that $Q=\emptyset$. This implies that $\tau(\delta)$ represents the same variable sequence for each $\delta>0$. Therefore, $\tau(\delta)=\tau(\gamma)$ for each $\delta>0$. Now, suppose that $Q\neq \emptyset$ and let $\hat{\delta}>0$ be given. When $0< \hat{\delta} < \delta_1$, then $\tau(\hat{\delta})=\tau(\gamma)$ since the variable sequence indicated by $\tau(\delta)$ does not change when $0<\delta<\delta_1$. When $\delta_1 \leq \hat{\delta}<\delta_m$, there exists $i\in M\setminus\{m\}$ such that $\delta_{i}\leq \hat{\delta}< \delta_{i+1}$. Since there are no two variables such that their relative order in $\tau(\delta)$ is reversed when $\delta_{i}\leq \delta<\delta_{i+1}$, $\tau(\delta)$ does not change in this range. Therefore, $\tau(\hat{\delta})=\tau(\delta_i)$. If $\hat{\delta}\geq \delta_m$, then $\tau(\hat{\delta})=\tau(\delta_m)$ since the variable sequence indicated by $\tau(\delta)$ does not change when $\delta_m\leq \delta$. Therefore, the result follows. \qed
\end{proof}
Proposition \ref{chap4:prop:equiv} states that $\tau(q)$'s for all $q\in \Delta$ can represent all distinct variable sequences indicated by $\tau(\delta)$. Therefore, there exists $q^{\star}\in \Delta$ such that $x(q^{\star})=x(\delta^{\star})$. Accordingly, the NCR can be solved
by comparing the objective values corresponding to $x(q)$'s for all $q\in \Delta$. Although $|\Delta|$ is $O(n^2)$, examining all $x(q)$'s may be time-consuming when $n$ is large. In the subsequent discussion, we show that $\Delta$ can be reduced by introducing upper and lower bounds for $\delta^\star$.

Recall that $\delta^{\star} =\sum_{j \in N}\sigma_j^2 x_j^{\star}$ where $x^{\star}$ is an optimal solution for the non-convex relaxation. Based on this definition, we define $\delta_U$ as
\begin{equation}
    \delta_{U} = \max_{x\in [0,1]^{n}}\left\{\sum_{j \in N}\sigma^{2}_j x_j: g(x)\leq b\right\}, \label{chap4:prob:boundmax}
\end{equation}
and $\delta_L$ as
\begin{equation}
    \delta_{L} = \min_{x\in [0,1]^{n}}\left\{\sum_{j \in N}\sigma^{2}_j x_j: g(x)\geq b\right\}.\label{chap4:prob:boundmin}
\end{equation}
$x^{\star}$ is feasible for both the problem \eqref{chap4:prob:boundmax} and \eqref{chap4:prob:boundmin} since $g(x^{\star})=b$ by Proposition \ref{chap4:prop:pack}. Therefore, $\delta_{U}$ and $\delta_L$ are upper and lower bounds for $\delta^{\star}$, respectively. 

\begin{prop}\label{prop:bound}
    $\delta_{U}$ and $\delta_L$ can be obtained in $O(n\log n)$, respectively.
\end{prop}
\begin{proof}
    The problem \eqref{chap4:prob:boundmax} is a special case of the CKP where $c_j=\sigma_j^2$ for each $j\in N$. Let us consider the NCR of this problem. By definition, the set $Q$ is empty since no two variables satisfy the condition \eqref{uniquecondition}. Accordingly, $\Delta=\{0,\gamma\}$ and the problem \eqref{chap4:prob:boundmax} can be solved by comparing $x(0)$ and $x(\gamma)$ which can be constructed in $O(n\log n)$. 
    
    Now, let us consider the problem \eqref{chap4:prob:boundmin}. Let $\tau_{\min}$ be the variable sequence of $N$ sorted in ascending order of $\sigma_j^{2}/a_j$. We can obtain $x(\tau_{\min})$ in $O(n\log n)$. We then show that $x(\tau_{\min})$ is an optimal solution for the problem \eqref{chap4:prob:boundmin}. Suppose that $x(\tau_{\min})$ is not optimal for the problem \eqref{chap4:prob:boundmin}. Then, there exists $\hat{x}$ such that $g(\hat{x})\geq b$ and $\sigma(\hat{x})<\sigma(x(\tau_{\min}))$ where $\sigma(x)=\sum_{j \in N}\sigma_j^2 x_j$. Let us consider the following problem.
\begin{align}
    \min\quad &\sum_{j \in N}\sigma^2_j x_j \nonumber\\
    \text{s.t}\quad &\sum_{j\in N}a_j x_j \geq b - \phiprob\sqrt{\sigma(x(\tau_{\min}))} \label{chap4:prob:subbound}\\
    & x \in [0,1]^n,\nonumber
\end{align}
This problem is the LP relaxation of a min-knapsack problem, which can be transformed into a fractional knapsack problem \citep{martello1990knapsack}. Accordingly, it can be easily shown that $x(\tau_{\min})$ is an optimal solution for the problem \eqref{chap4:prob:subbound}, where the optimal objective value is $\sigma(x(\tau_{\min}))$. On the other hand, $\hat{x}$ is also feasible for this problem because 
\begin{equation*}
    \sum_{j \in N}a_j \hat{x}_j \geq b-\phiprob\sqrt{\sigma(\hat{x})} > b-\phiprob\sqrt{\sigma(x(\tau_{\min}))} 
\end{equation*}
where the last inequality holds due to the assumption that $\sigma(\hat{x})<\sigma(x(\tau_{\min}))$. Furthermore, the objective value of the problem \eqref{chap4:prob:subbound} corresponding to $\hat{x}$ is $\sigma(\hat{x})$ which is less than $\sigma(x(\tau_{\min}))$. This result contradicts to the fact that $x(\tau_{\min})$ is an optimal solution for the problem \eqref{chap4:prob:subbound}. Therefore, such $\hat{x}$ cannot exist, and $x(\tau_{\min})$ is an optimal solution for the problem \eqref{chap4:prob:boundmin}. \qed
\end{proof}

Proposition \ref{prop:bound} states that the upper and lower bound for $\delta^{\star}$ can be obtained efficiently. By utilizing these bounds, we define $\Delta^{\star}$ as follows.
\begin{equation*}
    \Delta^{\star} = \{q\in \Delta: \delta_{L} \leq q \leq \delta_{U} \} \cup \{\max\{\delta_{L},\gamma\}\}
\end{equation*}
It is clear that $|\Delta^{\star}|\leq |\Delta|$. In addition, the following result holds. 
\begin{prop}\label{prop:equivopt}
    There exists $q^{\star}\in \Delta^{\star}$ such that $\tau(q^{\star})=\tau(\delta^\star)$.  
\end{prop}
\begin{proof}
   By Proposition \ref{chap4:prop:equiv}, there exists $q\in \Delta$ such that $\tau(q)=\tau(\delta^{\star})$. If $\delta_L\leq q\leq \delta_U$, then $q\in \Delta^{\star}$. Suppose that $q<\delta_L$. We note that $\gamma\leq q$ in this case since $q=0$ implies that $\delta^{\star}=\delta_L=0$ as shown in the proof of Proposition \ref{chap4:prop:equiv}. Then, $\max\{\delta_L,\gamma\}=\delta_L$ since $\gamma\leq q<\delta_L\leq \delta^{\star}$, that is, $\delta_L\in \Delta^{\star}$. On the other hand, $\tau(\delta_L)=\tau(\delta^{\star})$ since $\tau(\delta^{\star})=\tau(q)$ implies that $\tau(\delta)=\tau(\delta^{\star})$ for each $\delta$ such that $q\leq \delta\leq \delta^{\star}$. Now, suppose that $\delta_U<q$. Then, $\delta_L\leq \delta^{\star}<q$, which implies that $q=\gamma$ by the proof of Proposition \ref{chap4:prop:equiv}. In this case, $\max\{\delta_L,\gamma\}=\gamma$ and $q=\gamma \in \Delta^{\star}$. Therefore, the result follows. \qed
\end{proof}
Proposition \ref{prop:equivopt} implies that the NCR can be solved by comparing $x(q)$'s for all $q \in \Delta^{\star}$. Based on these results, we propose an efficient algorithm for solving the NCR, described in Algorithm \ref{chap4:alg:enum}. 
\begin{algorithm}[ht]
\caption{Algorithm for solving the NCR}
\small
\begin{algorithmic}[1]
\State $\delta_U,\delta_L\gets$ Solve the problem \eqref{chap4:prob:boundmax} and \eqref{chap4:prob:boundmin}, respectively ;
\State $\Delta^{\star} \gets \emptyset$, $\delta_1 \gets \infty$, and $z^{\star} \gets 0$ ;
\For{$\{k,l\} \subseteq N$}
\If{Condition \eqref{uniquecondition} is satisfied}
    \State $q\gets $ Solve $p_k(\delta)=p_l(\delta)$ ;
    \State $\delta_1\gets q$ if $q <\delta_1$ ;
    \State $\Delta^{\star} \gets \Delta^{\star} \cup \{q\}$ if $\delta_L \leq q \leq \delta_U$ ;
\EndIf
\EndFor
\State $\gamma \gets \alpha$ such that $0< \alpha<\delta_1$ ;
\State $\Delta^{\star} \gets \Delta^{\star} \cup \{\max\{\gamma,\delta_L\}\}$ ;
\For{$q \in \Delta^{\star}$}
\State Construct $\tau(q)$ and $x(q)$ ;
\State $z^{\star}\gets \sum_{j \in N}c_j x_j(q)$ if $z^{\star}<\sum_{j \in N}c_j x_j(q)$ ;
\EndFor
\State \Return $z^{\star}$ ;
\end{algorithmic}
\label{chap4:alg:enum}
\end{algorithm}

\begin{thm}\label{thm:poly}
The NCR can be solved in $O(n^3 \log n)$.
\end{thm}
\begin{proof}
    $\delta_L$ and $\delta_U$ can be obtained in $O(n\log n)$ by Proposition \ref{prop:bound}. Hence, $\Delta^{\star}$ can be obtained in $O(n^2)$, where $|\Delta^{\star}|\leq n^2$. $x(q)$ for each $q\in \Delta^{\star}$ can be constructed in $O(n\log n)$. Comparing the objective values corresponding to $x(q)$'s for all $q\in \Delta$ can be performed in $O(n^2)$. Therefore, the NCR can be solved in $O(n^3\log n)$. \qed
\end{proof}
Theorem \ref{thm:poly} states that the non-convex relaxation can be solved in polynomial time through the reformulation. 
In the following section, we show that an approximate solution for the CKP can be derived from the optimal solution for the non-convex relaxation, obtained through Algorithm \ref{chap4:alg:enum}.

\section{Polynomial-time 1/2-approximation algorithm for the CKP based on the non-convex relaxation}\label{chap4:sec:approx}

Here, let $x^{\star}$ be the optimal solution for the non-convex relaxation, obtained through Algorithm \ref{chap4:alg:enum}. By definition, $x^{\star}$ has at most one fractional variable, denoted as $t\in N$. We define $x^{D}$ and $x^{U}$ from $x(\delta^\star)$ in the same manner presented in the proof of Theorem \ref{chap4:prop:intgap}, i.e., $x_j^D=\lfloor x_j^{\star} \rfloor $ for each $j \in N$, while $x^{U}_t=1$ and $x_j^U=0$ for each $j\in N\setminus \{t\}$. Then, it is clear that $x^{D}$ and $x^{U}$ are feasible solutions for the CKP. 

Now, let us define $x^{A}\in\{0,1\}^n$ as follows.
\begin{equation*}
    x^{A}=
    \begin{cases}
        x^{D} & \text{, if } \sum_{j \in N}c_jx^{U}_j \leq \sum_{j \in N}c_jx^{D}_j \\
        x^{U} & \text{, otherwise.}
    \end{cases}
\end{equation*}
\begin{thm}\label{thm:approx}
    $x^{A}$ is a 1/2-approximate solution for the CKP, which can be obtained in $O(n^3\log n)$.
\end{thm}
\begin{proof}
    By the feasibility of $x^{D}$ and $x^{U}$ for the CKP, $x^{A}$ is also a feasible solution for the CKP. Recall that $z_{OPT}$ denotes the optimal objective value of the CKP. We can derive the following inequalities from the feasibility of $x^{D}$ and $x^{U}$.
    \begin{equation*}
        \sum_{j \in N}c_j x^{A}_j = \max\left\{\sum_{j \in N}c_jx^{U}_j, \sum_{j \in N}c_jx^{D}_j\right\} \leq \sum_{j \in N}c_jx^{U}_j+ \sum_{j \in N}c_jx^{D}_j \leq 2z_{OPT}
    \end{equation*}
    This result implies that the objective value corresponding to $x^{A}$ is less than or equal to $2z_{OPT}$. Therefore, $x^{A}$ is a 1/2-approximate solution for the CKP. On the other hand, $x(\delta^{\star})$ can be obtained in $O(n^3\log n)$ by Theorem \ref{thm:poly}, and $x^{A}$ can be derived from $x(\delta^{\star})$ in $O(n)$. Hence, this 1/2-approximate solution can be obtained in $O(n^3\log n)$. \qed
\end{proof}
Theorem \ref{thm:approx} directly leads us to devise a polynomial time 1/2-approximation algorithm for the CKP using Algorithm \ref{chap4:alg:enum}. The detail are described in Algorithm \ref{chap4:alg:approx}.
\begin{algorithm}[ht]
\caption{Polynomial time 1/2-approximation algorithm for CKP}
\small
\begin{algorithmic}[1]
\State $x^\star \gets$ Solve the non-convex relaxation using Algorithm \ref{chap4:alg:enum} ;
\State $t\gets$ fractional variable index in $x^\star$ ;
\State $x^{D}_j \gets \lfloor x^{\star}_j\rfloor $ for all $j\in N$ ;
\State $x^{U}_t \gets 1$ and $x^{U}_j \leftarrow 0$ for all $j\in N\setminus\{t\}$ ;
\If{$\sum_{j \in N}c_j x^{U}_j \leq \sum_{j \in N}c_j x^{D}_j $}
\State $x^{A}\gets x^{D}$ ;
\Else
\State $x^{A}\gets x^{U}$ ;
\EndIf
\State \Return $x^{A}$ ;
\end{algorithmic}
\label{chap4:alg:approx}
\end{algorithm}

In summary, we can obtain both quality-guaranteed upper and lower bounds for the CKP in polynomial time from the non-convex relaxation using Algorithm \ref{chap4:alg:enum} and \ref{chap4:alg:approx}. In the following section, we evaluate the efficiency of the proposed algorithms and the quality of the obtained bounds through computational experiments.


\section{Computational test results}\label{chap4:sec:exp}

In this section, we present our computational experiments on the non-convex relaxation and the proposed algorithms. First, we computationally compare three continuous relaxations for the CKP, introduced in Section \ref{chap4:sec:compare}. Subsequently, we evaluate the performance of the proposed approximation algorithm. The result was compared with other solution approaches for the CKP.
All algorithms were implemented using C++ with optimization solvers provided by XPRESS. Our tests were performed using a machine with an Intel Core i7, 3.10GHz CPU, and 16GB RAM.

Three types of the CKP instances were used in our experiments, which were generated in the same manner with \cite{han2016robust} as follows.
\begin{itemize}
    \item SC (strongly correlated): each $a_j$ is an integer value randomly chosen in $[1,100]$ and $c_j=a_j +100$.
    \item IC (inversely correlated): each $c_j$ is an integer value randomly chosen in $[1,100]$ and $a_j=\min\{100,c_j+10\}$.
    \item SS (subset sum): each $a_j$ is an integer value randomly chosen in $[1,100]$ and $c_j=a_j$.
\end{itemize}
The standard deviation $\sigma_j$ was randomly chosen in $[0.1a_j, 0.2a_j]$ for each $j\in N$, and we set $b=\lfloor \sum_{j\in N}a_j \rfloor $. For each combination of instance type (SC, IC, SS), $n\in \{100,500,1000,5000\}$, and $\rho\in\{0.9,0.95\}$, we generated 10 instances, and the average results were reported.


\subsection{Comparison of upper bounds provided by the continuous relaxations}
We computationally compared the upper bounds for the CKP, $z_C$, $z_P$, and $z_{NC}$ provided by the convex, non-convex, and polyhedral relaxations. The convex relaxation was solved using the SOCP solver provided by XPRESS with its default setting. We implemented the cutting-plane algorithm for the polyhedral relaxation. Specifically, we iteratively added the constraints \eqref{modconst} to the convex relaxation, where the convex relaxation with some of the constraints \eqref{modconst} was solved using the SOCP solver provided by XPRESS. The separation problem was solved using the greedy algorithm proposed by \cite{edmond1971}. Finally, the non-convex relaxation was solved using Algorithm \ref{chap4:alg:enum}. The time limit was 600 seconds. 

We present the average results specifically for $\rho=0.9$ in Table \ref{tab:result1}, as there was no significant difference in the outcomes when $\rho=0.95$. Here, ``time(s)'' denotes the average computation time required to solve each continuous relaxation. ``cgap($\%$)'' indicates the average duality gap closed by each continuous relaxation. It is calculated as 
\begin{equation*}
    \text{cgap($\%$)}=\frac{z_{C}-z}{z_{C}-z_{LB}}\times 100,
\end{equation*}
where $z$ denotes the objective value of each continuous relaxation, and $z_{LB}$ represents the best lower bound obtained by solving the ISP using XPRESS with its default setting within a time limit of 1800 seconds. A larger cgap implies a better quality of the upper bound. We marked the cgap value with ``$\star$'' for instances not solved within the time limit. The number of instances not solved is also given in parentheses. The column headed ``\#cut'' presents the average number of cutting planes added when solving the polyhedral relaxation. In addition, we reported $|\Delta|$ and $|\Delta^{\star}|$ investigated when solving the non-convex relaxation. 

\begin{table}[!th]
\setlength{\tabcolsep}{3pt}
  \centering
  \caption{Comparison of continuous relaxations for the CKP ($\rho=0.90$)}
    \label{tab:result1}
    \begin{tabular}{
    >{\centering}p{0.06\textwidth}
    >{\centering}p{0.05\textwidth}
    *{1}{>{\raggedleft}p{0.06\textwidth}}
    p{0.001\textwidth}
    *{1}{>{\raggedright}p{0.07\textwidth}}
    *{1}{>{\raggedright}p{0.06\textwidth}}
    *{1}{>{\raggedright}p{0.09\textwidth}}
    p{0.001\textwidth}
    *{1}{>{\raggedright}p{0.07\textwidth}}
    *{1}{>{\raggedright}p{0.09\textwidth}}
    *{1}{>{\raggedright}p{0.09\textwidth}}
    >{\raggedright\arraybackslash}p{0.07\textwidth}
    }
    \toprule
    \multirow{2}[0]{*}{Type} & \multirow{2}[0]{*}{$n$} & \multicolumn{1}{c}{$z_C$} & & \multicolumn{3}{c}{$z_P$} & & \multicolumn{4}{c}{$z_{NC}$} \\
    \cmidrule{3-3}
    \cmidrule{5-7}
    \cmidrule{9-12}
          &       &  \multicolumn{1}{l}{time(s)} &  & \multicolumn{1}{l}{cgap (\scriptsize{\%})} & \multicolumn{1}{l}{\#cut} &\multicolumn{1}{l}{time(s)} &  & \multicolumn{1}{l}{cgap(\scriptsize{\%})} & $|\Delta|$ & $|\Delta^{\star}|$ & \multicolumn{1}{l}{time(s)} \\
    \midrule
    \multirow{4}[0]{*}{SC} & \multicolumn{1}{r}{100} & 0.011  &       & 47.75  & 6.7   & 0.040  &       & 48.00  & $7.3\times10^2$   & $4.0\times10^1$  & 0.000  \\
          & \multicolumn{1}{r}{500} & 0.021  &       & 61.73  & 7.9   & 0.094  &       & 62.28  & $1.8\times 10^4$ & $5.8\times10^2$   & 0.022  \\
          & \multicolumn{1}{r}{1000} & 0.032  &       & 54.08  & 7.3   & 0.129  &       & 55.02  & $7.3\times 10^4$ & $1.7\times10^3$ & 0.147  \\
          & \multicolumn{1}{r}{5000} & 0.092  &       & 42.68  & 5.2   & 0.403  &       & 51.48  & $1.8\times 10^6$ & $1.8\times10^4$ & 9.583  \\
    \midrule
    \multirow{4}[0]{*}{IC} & \multicolumn{1}{r}{100} & 0.010  &       & 56.85  & 18.4  & 0.116  &       & 57.06  & $3.4\times 10^3$ & $4.7\times10^1$  & 0.000  \\
          & \multicolumn{1}{r}{500} & 0.011  &       & 61.80  & 16.4  & 0.227  &       & 62.13  & $8.5\times10^4$ & $5.2\times10^2$ & 0.017  \\
          & \multicolumn{1}{r}{1000} & 0.019  &       & 54.73  & 15.7  & 0.292  &       & 55.24  & $3.4\times10^5$ & $1.3\times10^3$ & 0.096  \\
          & \multicolumn{1}{r}{5000} & 0.062  &       & 41.19  & 23.2  & 1.748  &       & 43.05  & $8.7\times10^6$ & $1.2\times10^4$ & 4.497  \\
    \midrule
    \multirow{4}[0]{*}{SS} & \multicolumn{1}{r}{100} & 0.008  &       & 51.64  & 125.6 & 2.455  &       & 94.76  & 2     & 1     & 0.000  \\
          & \multicolumn{1}{r}{500} & 0.016  &       & 54.31$^{\star}$  & 562.0   & 362.3(6)  &  & 91.22  & 2     & 1     & 0.001  \\
          & \multicolumn{1}{r}{1000} & 0.021  &       & 56.58$^{\star}$  & 410.4 & 306.6(5)  &   & 91.16  & 2     & 1     & 0.002  \\
          & \multicolumn{1}{r}{5000} & 0.068  &       & 55.95  & 144.9 & 58.28  &       & 90.21  & 2     & 1     & 0.047  \\
    \midrule
    \multicolumn{2}{c}{average} & 0.031 &       & 53.27  &  112.0     & 61.06  &       & 66.80  & $9.2\times10^5$ & $2.8\times10^3$  & 1.201  \\
    \bottomrule
    \end{tabular}%
\end{table}






As discussed in Section \ref{chap4:sec:compare}, the non-convex relaxation consistently improved the duality gap compared to the polyhedral relaxation. The differences in bound improvements between polyhedral and non-convex relaxations were insignificant in SC and IC type instances, whereas the non-convex relaxation required more computation time when $n \geq 1000$. Nevertheless, even when $n=5000$, the non-convex relaxation could provide upper bounds for the CKP, which are tighter than $z_C$ and $z_P$, within 10 seconds. The lower and upper bounds for $\delta^{\star}$ significantly contributed to this result, as indicated in columns $|\Delta|$ and $|\Delta^{*}|$. These bounds substantially reduced the number of reverse points investigated in $\Delta$. Without these bounds, solving the non-convex relaxation took longer computation time, increasing proportionally to the ratio of $|\Delta^{\star}|$ to $|\Delta|$.

On the other hand, the non-convex relaxation outperformed the other continuous relaxations for SS type instances, both in terms of the quality of upper bounds and computation time. In this case, the cutting-plane algorithm for the polyhedral relaxation did not terminate within the time limit for some instances. However, the non-convex relaxation could be solved significantly faster, even compared to the convex relaxation. This result can be explained by the absence of reverse points for SS type instances, as $c_j=a_j$ for each $j\in N$, meaning $Q=\emptyset$.

\subsection{Performance of the proposed approximation algorithm}

We evaluated the performance of the proposed approximation algorithm for the CKP. The results were compared with the robust optimization-based heuristic algorithm proposed by \cite{joung2020robust} and the branch-and-bound algorithm provided by XPRESS, which we refer to as the RO heuristic and BnB, respectively. When we tested the RO heuristic, the parameter associated with the algorithm's step size was set to $K=3$ as recommended in \cite{joung2020robust}. The BnB solves the ISP using XPRESS with its default setting. The time limit for the BnB was set to 1800 seconds.

We also implemented the robust optimization-based approach by \cite{han2016robust} and PTAS by \cite{goyal2010ptas}. For the algorithm by \cite{han2016robust}, the parameter was set to as recommended in their paper. We set $\epsilon =1/6$ for the PTAS to obtain a $1/2$-approximate solution. However, these algorithms failed to find a feasible solution even after 2 hours, as observed by \cite{joung2020robust}. Hence, we did not report the results of these algorithms. 

Table \ref{tab:result2} shows the results of the proposed algorithm, RO heuristic, and the BnB. The column headed ``time (s)'' denotes the computation time for each algorithm. ``gap ($\%$)'' refers to the duality gaps of the obtained objective values, computed as follows.
\begin{equation*}
    \text{gap ($\%$)}=\frac{z_{UB}-\hat{z}}{z_{UB}}\times 100
\end{equation*}
Here, $\hat{z}$ denotes the objective values corresponding to the feasible solutions found in each algorithm. $z_{UB}$ indicates the best upper bound, which is the smaller of the upper bounds obtained from the BnB and non-convex relaxation. A smaller gap implies a better feasible solution. When the average gap value of the BnB result is greater than $0$, we presented the number of instances not solved to optimality in parentheses.

\begin{table}[!th]
\setlength{\tabcolsep}{4pt}
  \centering
  \caption{Comparison of the proposed approximation algorithm and other solution approaches for the CKP ($\rho=0.9$)}
    \label{tab:result2}
    \begin{tabular}{
    >{\centering}p{0.06\textwidth}
    >{\centering}p{0.05\textwidth}
    *{1}{>{\raggedright}p{0.06\textwidth}}
    *{1}{>{\raggedright}p{0.06\textwidth}}
    >{\centering}p{1mm}
    *{1}{>{\raggedright}p{0.06\textwidth}}
    *{1}{>{\raggedright}p{0.06\textwidth}}
    >{\centering}p{1mm}
    *{1}{>{\raggedright}p{0.06\textwidth}}
    *{1}{>{\raggedright\arraybackslash}p{0.12\textwidth}}
    }
    \toprule
    \multirow{2}[0]{*}{Type} & \multirow{2}[0]{*}{$n$} & \multicolumn{2}{c}{Proposed} & & \multicolumn{2}{c}{RO heuristic} & & \multicolumn{2}{c}{BnB} \\
    \cmidrule{3-4}
    \cmidrule{6-7}
    \cmidrule{9-10}
          &       & \multicolumn{1}{l}{gap(\%)} & \multicolumn{1}{l}{time(s)} &  & \multicolumn{1}{l}{gap(\%)} & \multicolumn{1}{l}{time(s)} &  & \multicolumn{1}{l}{gap(\%)} & \multicolumn{1}{l}{time(s)} \\
    \midrule
    \multirow{4}[0]{*}{SC} & \multicolumn{1}{r}{100} & 1.357  & 0.001  &       & 0.013  & 0.080  &       & 0  & 69.564  \\
          & \multicolumn{1}{r}{500} & 0.247  & 0.023  &       & 0.014  & 0.855  &       & 0.016  & 920.08(5)  \\
          & \multicolumn{1}{r}{1000} & 0.103  & 0.149  &       & 0.017  & 3.624  &       & 0.011  & 1462.7(8) \\
          & \multicolumn{1}{r}{5000} & 0.032  & 9.621  &       & 0.005  & 98.12  &       & 0.004  & 1800.0(10) \\
    \midrule
    \multirow{4}[0]{*}{IC} & \multicolumn{1}{r}{100} & 1.360  & 0.001  &       & 0.004  & 0.063  &       & 0  & 40.523  \\
          & \multicolumn{1}{r}{500} & 0.279  & 0.017  &       & 0.007  & 0.452  &       & 0.009  & 563.12(3) \\
          & \multicolumn{1}{r}{1000} & 0.144  & 0.096  &       & 0.006  & 1.477  &       & 0.009  & 1083.0(6) \\
          & \multicolumn{1}{r}{5000} & 0.016  & 4.560  &       & 0.002  & 27.73  &       & 0.004  & 1621.6(9) \\
    \midrule
    \multirow{4}[0]{*}{SS} & \multicolumn{1}{r}{100} & 1.494  & 0.001  &       & 0.031  & 0.059  &       & 0.017  & 366.48(2)  \\
          & \multicolumn{1}{r}{500} & 0.311  & 0.001  &       & 0.021  & 0.344  &       & 0.086  & 1800.0(10) \\
          & \multicolumn{1}{r}{1000} & 0.135  & 0.002  &       & 0.015  & 0.864  &       & 2.198  & 1800.0(10) \\
          & \multicolumn{1}{r}{5000} & 0.037  & 0.047  &       & 0.007  & 15.54  &       & 23.81  & 1800.0(10) \\
    \midrule
    \multicolumn{2}{c}{average} & 0.459  & 1.204  &       & 0.012  & 12.43  &       & 2.180  & 1110.6  \\
    \bottomrule
    \end{tabular}%
\end{table}

As discussed in Section \ref{chap4:sec:approx}, the proposed approximation algorithm guarantees that the integrality gaps of the obtained solutions are less than or equal to $50\%$. However, in our experiments, the observed duality gaps were significantly below $50\%$ for all instances, although the duality gaps are greater than or equal to the integrality gaps by definition. 

The RO heuristic provided near-optimal solutions with gaps of less than $0.05\%$ for all instances within about 100 seconds. Furthermore, for some instances of IC and SS types, the solutions obtained through the RO heuristic outperformed those obtained by the BnB within the time limit.

The solutions obtained through our approximation algorithm were not better than those found in the RO heuristic. However, our algorithm found feasible solutions in significantly less time than the RO heuristic, while the duality gaps of the obtained solutions were not large ($<0.5\%$ on average). Furthermore, our algorithm outperformed the BnB for large SS type instances, both in terms of the quality of solutions and computation time.
Therefore, our algorithm may be more effective than the RO heuristic when solving numerous CKPs, such as sub-problems of large-scale integer programs. Additionally, our algorithm ensures solution quality and does not require tuning parameters, in contrast to the RO heuristic.


\section{Conclusion}\label{chap4:sec:con}
In this study, we introduced a non-convex relaxation for the CKP and compared it with other continuous relaxations. We showed that the non-convex relaxation provides an upper bound for the CKP, at least as tight as those obtained from the other continuous relaxations. We also showed that the non-convex relaxation provides both upper and lower bounds for the CKP in polynomial time, where the quality of the bounds is ensured. The computational test results demonstrated that these upper and lower bounds are tight for the CKP and obtainable within a short time. The proposed approximation algorithm may be utilized as a heuristic algorithm for solving numerous CKPs, such as sub-problems of large-scale integer programs. Furthermore, the upper bounds provided by the non-convex relaxation can be used in the lifting technique for probabilistic cover inequalities for the CKP, proposed by \cite{atamturk2009submodular}. To enhance the effectiveness of our results in these applications, developing more efficient solution approach for the non-convex relaxation is necessary. It is also an interesting research direction to develop solution approaches based on continuous relaxations for the CKP under different probability distributions.





\bibliographystyle{spbasic}

\begin{thebibliography}{24}
\providecommand{\natexlab}[1]{#1}
\providecommand{\url}[1]{{#1}}
\providecommand{\urlprefix}{URL }
\expandafter\ifx\csname urlstyle\endcsname\relax
  \providecommand{\doi}[1]{DOI~\discretionary{}{}{}#1}\else
  \providecommand{\doi}{DOI~\discretionary{}{}{}\begingroup
  \urlstyle{rm}\Url}\fi
\providecommand{\eprint}[2][]{\url{#2}}

\bibitem[{Atamt{\"u}rk and Narayanan(2008)}]{atamturk2008polymatroids}
Atamt{\"u}rk A, Narayanan V (2008) Polymatroids and mean-risk minimization in
  discrete optimization. Operations Research Letters 36(5):618--622

\bibitem[{Atamt{\"u}rk and Narayanan(2009)}]{atamturk2009submodular}
Atamt{\"u}rk A, Narayanan V (2009) The submodular knapsack polytope. Discrete
  Optimization 6(4):333--344

\bibitem[{Ben-Tal and Nemirovski(1998)}]{ben1998robust}
Ben-Tal A, Nemirovski A (1998) Robust convex optimization. Mathematics of
  operations research 23(4):769--805

\bibitem[{Ben-Tal et~al.(2009)Ben-Tal, El~Ghaoui, and
  Nemirovski}]{ben2009robust}
Ben-Tal A, El~Ghaoui L, Nemirovski A (2009) Robust optimization, vol~28.
  Princeton university press

\bibitem[{Bertsimas and Popescu(2005)}]{bertsimas2005optimal}
Bertsimas D, Popescu I (2005) Optimal inequalities in probability theory: A
  convex optimization approach. SIAM Journal on Optimization 15(3):780--804

\bibitem[{Bertsimas and Sim(2004)}]{bertsimas2004price}
Bertsimas D, Sim M (2004) The price of robustness. Operations research
  52(1):35--53

\bibitem[{Birge and Louveaux(2011)}]{birge2011introduction}
Birge JR, Louveaux F (2011) Introduction to stochastic programming. Springer
  Science \& Business Media

\bibitem[{Charnes et~al.(1958)Charnes, Cooper, and Symonds}]{charnes1958cost}
Charnes A, Cooper WW, Symonds GH (1958) Cost horizons and certainty
  equivalents: an approach to stochastic programming of heating oil. Management
  science 4(3):235--263

\bibitem[{Conforti et~al.(2014)Conforti, G{\'e}rard, and
  Zambelli}]{cornuelsols2014IP}
Conforti M, G{\'e}rard C, Zambelli G (2014) Integer Programming. Springer
  International Publishing

\bibitem[{Dantzig(1957)}]{dantzig1957discrete}
Dantzig GB (1957) Discrete-variable extremum problems. Operations research
  5(2):266--288

\bibitem[{Edmond(1970)}]{edmond1971}
Edmond J (1970) Submodular functions, matroids, and certain polyhedra. In:
  Combinatorial Structures and Their Applications, Gordon and Breach, Louvain,
  pp 69--87

\bibitem[{Ghaoui et~al.(2003)Ghaoui, Oks, and Oustry}]{ghaoui2003worst}
Ghaoui LE, Oks M, Oustry F (2003) Worst-case value-at-risk and robust portfolio
  optimization: A conic programming approach. Operations research
  51(4):543--556

\bibitem[{Gilmore and Gomory(1961)}]{gilmore1961linear}
Gilmore PC, Gomory RE (1961) A linear programming approach to the cutting-stock
  problem. Operations research 9(6):849--859

\bibitem[{Goyal and Ravi(2010)}]{goyal2010ptas}
Goyal V, Ravi R (2010) A ptas for the chance-constrained knapsack problem with
  random item sizes. Operations Research Letters 38(3):161--164

\bibitem[{Han et~al.(2016)Han, Lee, Lee, Choi, and Park}]{han2016robust}
Han J, Lee K, Lee C, Choi KS, Park S (2016) Robust optimization approach for a
  chance-constrained binary knapsack problem. Mathematical Programming
  157(1):277--296

\bibitem[{Iida(1999)}]{iida1999note}
Iida H (1999) A note on the max-min 0-1 knapsack problem. Journal of
  Combinatorial Optimization 3:89--94

\bibitem[{Joung and Lee(2020)}]{joung2020robust}
Joung S, Lee K (2020) Robust optimization-based heuristic algorithm for the
  chance-constrained knapsack problem using submodularity. Optimization Letters
  14(1):101--113

\bibitem[{Kellerer et~al.(2004)Kellerer, Pferschy, and Pisinger}]{kellerer2004}
Kellerer H, Pferschy U, Pisinger D (2004) Knapsack Problems. Springer, Berlin,
  Germany

\bibitem[{Kosuch and Lisser(2010)}]{kosuch2010upper}
Kosuch S, Lisser A (2010) Upper bounds for the 0-1 stochastic knapsack problem
  and a b\&b algorithm. Annals of Operations Research 176:77--93

\bibitem[{Martello and Toth(1990)}]{martello1990knapsack}
Martello S, Toth P (1990) Knapsack problems: algorithms and computer
  implementations. John Wiley \& Sons, Inc.

\bibitem[{Merzifonluo{\u{g}}lu et~al.(2012)Merzifonluo{\u{g}}lu, Geunes, and
  Romeijn}]{merzifonluouglu2012static}
Merzifonluo{\u{g}}lu Y, Geunes J, Romeijn HE (2012) The static stochastic
  knapsack problem with normally distributed item sizes. Mathematical
  Programming 134(2):459--489

\bibitem[{Savelsbergh(1997)}]{savelsbergh1997branch}
Savelsbergh M (1997) A branch-and-price algorithm for the generalized
  assignment problem. Operations research 45(6):831--841

\bibitem[{Shabtai et~al.(2018)Shabtai, Raz, and Shavitt}]{shabtai2018relaxed}
Shabtai G, Raz D, Shavitt Y (2018) A relaxed fptas for chance-constrained
  knapsack. In: 29th International Symposium on Algorithms and Computation
  (ISAAC 2018), Schloss Dagstuhl-Leibniz-Zentrum fuer Informatik

\bibitem[{Yu(1996)}]{yu1996max}
Yu G (1996) On the max-min 0-1 knapsack problem with robust optimization
  applications. Operations Research 44(2):407--415

\end{thebibliography}

\end{document}